\documentclass{amsart}

\usepackage{amssymb}
\usepackage{amsmath}
\usepackage{amsthm}
\usepackage{amsbsy}
\usepackage{bm}
\usepackage{hyperref}
\date{\today}
\usepackage{cite}
\usepackage{array}
\usepackage{xcolor}
\usepackage{tikz}
\usepackage{graphics}
\usepackage{subcaption}
\theoremstyle{theorem}
    \newtheorem{theorem}{Theorem}
    \newtheorem{lemma}[theorem]{Lemma}

\theoremstyle{definition} 
    \newtheorem{definition}[theorem]{Definition}
    
    \newtheorem{result}[theorem]{Result}
    \newtheorem{remark}[theorem]{Remark}
    \newtheorem{example}[theorem]{Example}
    \newtheorem{exercise}[theorem]{Exercise}

\def\suchthat{\; : \;}

\def\tends{\rightarrow}

\def\l{\left}
\def\r{\right}
\def\<{\langle}
\def\>{\rangle}

\newcommand{\E}{\mbox{E}}

\newcommand\Tr{{\mbox{Tr}}}

\newcommand\mnote[1]{} 
\newcommand\be{\begin{equation*}}

\newcommand\ee{\end{equation*}}

\newcommand\ben{\begin{equation}}
\newcommand\een{\end{equation}}
\newcommand\bes{\begin{eqnarray*}}
\newcommand\ees{\end{eqnarray*}}

\newcommand\bex{\begin{exercise}}
\newcommand\eex{\end{exercise}}
\newcommand\beg{\begin{example}}
\newcommand\eeg{\end{example}}
\newcommand\benu{\begin{enumerate}}
\newcommand\eenu{\end{enumerate}}
\newcommand\beit{\begin{itemize}}
\newcommand\eeit{\end{itemize}}
\newcommand\berk{\begin{remark}}
\newcommand\eerk{\end{remark}}
\newcommand\bdefn{\begin{defintion}}
\newcommand\edefn{\end{definition}}
\newcommand\bthm{\begin{theorem}}
\newcommand\ethm{\end{theorem}}
\newcommand\bprf{\begin{proof}}
\newcommand\eprf{\end{proof}}
\newcommand\blem{\begin{lemma}}
\newcommand\elem{\end{lemma}}

\newcommand{\Cov}{\mbox{\rm Cov}}
\newcommand{\sm}{{\raise0.3ex\hbox{$\scriptstyle \setminus$}}}

\def\l{\left}
\def\r{\right}

\def\tends{\rightarrow}

\def\CHI{\mathchoice%
{\raise2pt\hbox{$\chi$}}%
{\raise2pt\hbox{$\chi$}}%
{\raise1.3pt\hbox{$\scriptstyle\chi$}}%
{\raise0.8pt\hbox{$\scriptscriptstyle\chi$}}}
\def\smalloplus{\raise1pt\hbox{$\,\scriptstyle \oplus\;$}}

\title[linear eigenvalue statistics of reverse circulant matrices]{Fluctuation of linear eigenvalue statistics  of reverse circulant matrices with  independent entries}

\author{Shambhu Nath Maurya}
\address{Department of Mathematics\\
	Indian Institute of Technology Bombay\\
	Powai, Mumbai, Maharashtra 400076, India}

\email{snmaurya [at] math.iitb.ac.in}

\author{Koushik Saha}
\address{Department of Mathematics\\
        Indian Institute of Technology Bombay\\
         Powai, Mumbai, Maharashtra 400076, India}

\email{koushik.saha [at] iitb.ac.in}

\date{\today}
\thanks{The work of Shambhu Nath Maurya is partially supported by UGC Doctoral Fellowship, India and the work of Koushik Saha is partially supported by MATRICS grant of SERB, Department of Science and Technology, Government of India.}
\begin{document}

\begin{abstract}
In this article, we study the fluctuations of linear eigenvalue statistics of reverse circulant $(RC_n)$ matrices  with independent entries which satisfy some moment conditions.
We show that $\frac{1}{\sqrt{n}} \Tr \phi(RC_n)$ obey the central limit theorem (CLT) type result, where 
$\phi$ is a nice test function.
\end{abstract}

\maketitle

\section{ introduction and main results}
Let $A_n$ be an $n\times n$ matrix with real or complex entries. The linear statistics of
eigenvalues $\lambda_1,\lambda_2,\ldots, \lambda_n$ of $A_n$ is a
function of the form
\begin{equation} \label{eqn:1}
\frac{1}{n}\sum_{k=1}^{n}\phi(\lambda_k)
\end{equation}
where $\phi$ is some  fixed function. The function $\phi$ is known as the test function. The  fluctuation of linear statistics of eigenvalues is one of the main  
area of research in random matrix theory.  


First, it was studied by Arharov \cite{arharov} in 1971 for sample covariance matrices. 
From then  the fluctuations of
eigenvalues for various  random matrices have been extensively studied by various people.  
For  results on fluctuations of linear eigenvalue statistics of  Wigner and sample covariance matrices,  see  \cite{johansson1998}, \cite{soshnikov1998tracecentral}, \cite{bai2004clt}, \cite{lytova2009central}, \cite{shcherbina2011central}.  For   band and sparse  random matrices, see  \cite{anderson2006clt},  \cite{jana2014}, \cite{li2013central},  \cite{shcherbina2015} and for Toeplitz and band Toeplitz matrices, see  \cite{chatterjee2009fluctuations} and \cite{liu2012}.

In a recent article \cite{adhikari_saha2017},  the CLT for linear eigenvalue statistics has been established in total variation norm for   reverse circulant matrices   with Gaussian entries and even degree monomial test functions. In a subsequent article \cite{adhikari_saha2018}, the authors extended their results for independent entries which are smooth functions of Gaussian variables.  Here we consider the fluctuation problem for  reverse circulant matrices with general entries which are independent and  satisfy some moment condition. However we could deal only with the test functions which are even degree monomials  or polynomials with even degree terms. 

A sequence is said to be an {\it input sequence} if the matrices are constructed from the given sequence. We consider the input sequence of the form $\{x_i: i\geq 1\}$
and the reverse circulant matrix is defined as 
$$
RC_n=\left(\begin{array}{cccccc}
x_1 & x_2 & x_3 & \cdots & x_{n-1} & x_n \\
x_2 & x_3 & x_4 & \cdots & x_{n} & x_{1}\\
x_3 & x_4 & x_5 & \cdots & x_{1} & x_{2} \\
\vdots & \vdots & {\vdots} & \ddots & {\vdots} & \vdots \\
x_n & x_1 & x_2 & \cdots & x_{n-2} & x_{n-1}
\end{array}\right).
$$
For $j=1,2,\ldots, n-1$, its $(j+1)$-th row is obtained by giving its $j$-th row a left circular shift by one positions. Note that the matrix is symmetric and the (i,\;j)-th element of the matrix is $x_{(i+j-1) \bmod  n}$.

Now we consider linear eigenvalue statistics as defined in \eqref{eqn:1} for $RC_n$ with test function $\phi(x)=x^{2p}$, $p\geq 1$. Therefore
$$ \sum_{k=1}^{n} \phi (\lambda_k)= \sum_{k=1}^{n}(\lambda_k)^{2p}= \Tr(RC_n)^{2p},$$
where $\lambda_1,\lambda_2,\ldots,\lambda_n$ are the eigenvalues of $RC_n$. We scale and centre $\Tr(RC_n)^{2p}$ to study its fluctuation, and define 
\begin{equation}\label{eqn:RCw_p}
w_p := \frac{1}{\sqrt{n}} \bigl\{ \Tr(RC_n)^{2p} - \E[\Tr(RC_n)^{2p}]\bigr\}. 
\end{equation}
Moreover, for a given polynomial  $$Q(x)=\sum_{k=1}^da_kx^{2k}$$ 
with degree $2d$, $d\geq 1$ and with even degree terms, we define 
\begin{equation}\label{eqn:RCw_Q}
w_Q := \frac{1}{\sqrt{n}} \bigl\{ \Tr(Q(RC_n)) - \E[\Tr(Q(RC_n))]\bigr\}. 
\end{equation}
 
Note that $w_Q$ and $w_p$ depends on $n$. But we suppress $n$  to keep the notation simple. In our first result, we calculate the covariance between $w_p$ and $w_q$ as $n \tends \infty$. 
\begin{theorem}\label{thm:revcircovar}
 Suppose  $RC_n$ is the reverse circulant matrix with independent input sequence $\{\frac{x_i}{\sqrt n}\}_{i\geq 1}$ such that
 \begin{equation}\label{eqn:condition}
\E(x_i)=0, \E(x_i^2)=1,  \E(x_i^4)=\E(x_1^4) \ \forall \ i \geq 1 \mbox{ and}\  \sup_{i\geq 1}\E(|x_i|^k)=\alpha_k<\infty \ \mbox{for}\ k\geq 3.
\end{equation}  
 Then for $p,q \geq1$,
\begin{equation} \label{eqn:sigma_p,q}
\sigma_{p,q}:=\lim_{n\to \infty} \Cov\big(w_p,w_q\big) = \sum_{k=2}^{\min\{p,q\}}c_k g(k) + (\E x_1^4-1)c_1,
\end{equation}
where 
\begin{align*}
c_k &=\binom{p}{p-k}^2(p-k)! \binom{q}{q-k}^2(q-k)!,\\
g(k)&=\frac{1}{(2k-1)!}\sum_{s=-(k-1)}^{k-1}\sum_{j=0}^{k+s-1}(-1)^j\binom{2k}{j}(k+s-j)^{2k-1}(2-{\bf 1}_{\{s=0\}})k!k!.
\end{align*}
Moreover, if $p=q$ then we denote $\sigma_{p,q}$ by $\sigma^{2}_{p}$.
\end{theorem}
In our second result, we see the fluctuation of linear  eigenvalue statistics of  reverse circulant matrices with polynomial test functions with even degree terms. 

\begin{theorem}\label{thm:revcirpoly} 
Suppose  input entry of $RC_n$ is independent sequence $\{\frac{x_i}{\sqrt n}\}_{i\geq 1}$ which satisfy (\ref{eqn:condition}).
Then, as $n\to \infty$,
\begin{align*}
w_Q \stackrel{d}{\longrightarrow} N(0,\sigma_{Q}^2),
\end{align*}
where $\displaystyle \sigma_{Q}^2= \sum_{\ell=1}^d \sum_{k=1}^d a_{\ell} a_k \sigma_{\ell,k}$ and $\sigma_{\ell,k}$ is as in (\ref{eqn:sigma_p,q}). In particular, for $Q(x)= x^{2p}$,
\begin{align*}
 w_p \stackrel{d}{\longrightarrow} N(0,\sigma_{p}^2),
\end{align*}
where $\sigma^2_{p}= \sigma_{p,p}$.
\end{theorem}

\begin{remark}
	In the above theorems we have considered the fluctuation of $w_p$ for  $p\geq1$. For $p=0$, 
	\begin{align*}
	w_0 = \frac{1}{\sqrt{n}} \bigl\{ \Tr(I) - \E[\Tr(I)]\bigr\} = \frac{1}{\sqrt{n}}[n-n] = 0
	\end{align*}
	and hence it has no fluctuation.
	So we ignore this case, that is, for $p=0$.
\end{remark}

In Section \ref{sec:cov} we prove Theorem \ref{thm:revcircovar}. We derived trace formula and state some results which will be used to prove Theorem \ref{thm:revcircovar}. In Section \ref{sec:poly} we prove Theorem \ref{thm:revcirpoly}. We use  method of moments and  Wick's formula to prove Theorem  \ref{thm:revcirpoly}.

\section{Proof of Theorem \ref{thm:revcircovar}}\label{sec:cov}

We first define some notation which will be used   in the proof of Theorem \ref{thm:revcirpoly}.
\begin{align}\label{def:A_2p}
A_{2p}&=\{(i_1,\ldots,i_{2p})\in \mathbb N^{2p}\suchthat \sum_{k=1}^{2p}(-1)^ki_k=0 \mbox{ (mod $n$) }, 1\le i_1,\ldots,i_{2p}\le n\},\\
A_{2p}'&=\{(i_1,\ldots,i_{2p})\in \mathbb N^{2p}\suchthat \sum_{k=1}^{2p}(-1)^ki_k=0 \mbox{ (mod $n$) }, 1\le i_1\neq i_2\neq \cdots\neq i_{2p}\le n\}, \nonumber\\
A_{2p,s}&=\{(i_1,\ldots,i_{2p})\in \mathbb N^{2p}\suchthat \sum_{k=1}^{2p}(-1)^ki_k=sn,1\le i_1,\ldots,i_{2p}\le n\}, \nonumber\\
A_{2p,s}'&=\{(i_1,\ldots,i_{2p})\in \mathbb N^{2p}\suchthat \sum_{k=1}^{2p}(-1)^ki_k=sn,1\le i_1\neq i_2\neq \cdots\neq i_{2p}\le n\}. \nonumber
\end{align}
Now recall the trace formula for reverse circulant matrices from \cite{adhikari_saha2017}.
Let $e_1,\ldots,e_n$ be the standard unit vectors in $\mathbb R^n$, i.e., $e_i=(0,\ldots,1,\ldots, 0)^t$ ($1$ in $i$-th place).  Therefore we have 
	\begin{align*}
	(RC_n)e_i=\mbox{$i$-th column}=\sum_{i_1=1}^nX_{i_1}e_{i_1-i+1 \mbox{ mod $n$}},
	\end{align*}
	for $i=1,\ldots, n$ (we write $e_0=e_n$). Repeating the procedure we get 
	\begin{align*}
	(RC_n)^2e_i=\sum_{i_1,i_2=1}^nX_{i_1}X_{i_2}e_{i_2-i_1+i \mbox{ mod $n$}},
	\end{align*}
	for $i=1,\ldots, n$. Therefore in general we get
	\begin{align*}
	(RC_n)^{2p}e_i&=\sum_{i_1,\ldots,i_{2p}=1}^nX_{i_1}\ldots X_{i_{2p}}e_{i_{2p}-i_{2p-1}\cdots -i_1+i \mbox{ mod $n$}},
	\\(RC_n)^{2p+1}e_i&=\sum_{i_1,\ldots,i_{2p+1}=1}^nX_{i_1}\ldots X_{i_{2p+1}}e_{i_{2p+1}-i_{2p}\cdots +i_1-i+1 \mbox{ mod $n$}},
	\end{align*}
	for $i=1,\ldots, n$.  Therefore the trace of $(RC_n)^{2p}$ can be written as 
	\begin{align}\label{trace formula RC_n}
	\Tr[(RC_n)^{2p}]=\sum_{i=1}^ne_i^t(RC_n)^{2p}e_i=n\sum_{A_{2p}}X_{i_1}\ldots X_{i_{2p}},
	\end{align}
where $A_{2p}$ as defined in \eqref{def:A_2p}.

The following result will be used in the proof of  Theorem \ref{thm:revcircovar}. 
\begin{result} \label{result:cardinalityA}
	Suppose $|A_{2p,s}|$ denotes   the cardinality of $A_{2p,s}$. Then 
	$$
	|A_{2p,s}|=\sum_{k=0}^{p+s-1}(-1)^k\binom{2p}{k}\binom{(p+s-k)n+p-1}{2p-1}, \;\;\mbox{ for $s=-(p-1),\ldots, 0,\ldots, p-1$.}
	$$ 
	
\end{result}
For the proof of Result \ref{result:cardinalityA}, we refer to \cite[Lemma 14]{adhikari_saha2017}. 
Now for a given vector $(i_1, i_2, \ldots, i_p)$, we define {\it odd-even pair matched} elements of the vector. 
\begin{definition}\label{def:odd-even}
	Suppose $(i_1, i_2, \ldots, i_p)$ is a vector in $\mathbb N^p$. Two elements $i_k, i_\ell$ of $(i_1, i_2, \ldots, i_p)$  are said to be {\it odd-even pair matched} if $i_k=i_\ell$ and one of the elements  appears at an odd position in $(i_1, i_2, \ldots, i_p)$ and other one appears at an even position in $(i_1, i_2, \ldots, i_p)$. For example, in $(1,1,3,4)$, entry $1$ is {\it odd-even pair matched} whereas in $(1,2,1,3)$, entry $1$ is  not {\it odd-even pair matched}.
	
	We shall call a vector  $(i_1, i_2, \ldots, i_p)$ is {\it odd-even pair matched} if all the entries of $(i_1, i_2, \ldots, i_p)$ are {\it odd-even pair matched}. In other words, each entry of $(i_1, i_2, \ldots, i_p)$ appears same number of times in odd and even positions. 
\end{definition}
Observe that, if $(i_1,i_2,\ldots,i_{2p})\in A_{2p}$  and each entry of $\{i_1,i_2,\ldots,i_{2p} \}$ has multiplicity greater than or equal to two, then the maximum number of free variables in $(i_1,i_2,\ldots,i_{2p})$ will be $p$, only when $(i_1,i_2,\ldots,i_{2p})$ is {\it odd-even pair matched}. By free variables we mean that these variables can be chosen freely from their range $\{1,2,\ldots,n\}$. We shall use this observation in the proof of Theorem \ref{thm:revcircovar}. 

Now we  prove Theorem \ref{thm:revcircovar}. 
\begin{proof}[Proof of Theorem \ref{thm:revcircovar}]
	Since $\E(w_p)=\E(w_q)=0$ therefore we get
	\begin{align*}
	\Cov\big(w_p,w_q\big)&=\E[w_p w_q]\\
	&= \frac{1}{n} \Big\{ \E[\Tr(RC_n)^{2p}\Tr(RC_n)^{2q} ]- \E[\Tr(RC_n)^{2p}]\E[\Tr(RC_n)^{2q}]   \Big\}. 
	\end{align*}
	Using the trace formula (\ref{trace formula RC_n}), we get
	\begin{align*}
	\E[\Tr(RC_n)^{2p}]&= \E\Big[n \sum_{A_{2p}} \frac{x_{i_1}}{ \sqrt{n}} \frac{x_{i_2}}{ \sqrt{n}}\cdots \frac{x_{i_{2p}}}{ \sqrt{n}} \Big]
	= \frac{1}{ n^ {p-1}} \E\Big[ \sum_{A_{2p}} x_{i_1} x_{i_2} \cdots x_{i_{2p}}\Big].	
	\end{align*}
	Therefore
	\begin{align} \label{eqn:T_1+T_2_RC}
	\Cov\big(w_p,w_q\big)&=\E[w_p w_q] \nonumber\\
	 &=\frac{1}{n^{p+q-1}} \sum_{A_{2p}, A_{2q}} \Big\{\E[x_{i_1}x_{i_2}\cdots x_{i_{2p}} x_{j_1} x_{j_2}\cdots x_{j_{2q}} ]
	 -\E[ x_{i_1} x_{i_2}\cdots x_{i_{2p}} ] \E[x_{j_1} x_{j_2} \cdots x_{j_{2q}}  ]        \Big\}.
	\end{align}
	First observe that, if $\{i_1,i_2,\ldots,i_{2p}\}\cap \{j_1,j_2,\ldots,j_{2p}\}=\emptyset$ then from independence of $x_i$'s
	$$\E[x_{i_1}x_{i_2}\cdots x_{i_{2p}} x_{j_1} x_{j_2}\cdots x_{j_{2q}} ]
	 -\E[ x_{i_1} x_{i_2}\cdots x_{i_{2p}} ] \E[x_{j_1} x_{j_2} \cdots x_{j_{2q}}  ]   =0.$$
	Therefore we can get the non-zero contribution in (\ref{eqn:T_1+T_2_RC}) only when there is at least one cross matching among $\{i_1, i_2,\ldots,i_{2p}\}$ and $\{j_1, j_2, \ldots,j_{2q}\}$, i.e., $\{i_1,i_2,\ldots,i_{2p}\}\cap \{j_1,j_2,\ldots,j_{2q}\}\neq\emptyset$. 
	
	Now using the above observation, (\ref{eqn:T_1+T_2_RC}) can be written as
	\begin{align}\label{eqn:I_k_RC}
	\lim_{n\to\infty} \Cov\big(w_p,w_q\big) 
	=& \lim_{n\to\infty}\frac{1}{n^{p+q-1}}\sum_{m=1}^{\min\{2p,2q\}}\sum_{I_{m}} \big( \E[x_{i_1}x_{i_2}\cdots x_{i_{2p}} x_{j_1} x_{j_2}\cdots x_{j_{2q}} ] \\
	&\qquad -\E[ x_{i_1} x_{i_2}\cdots x_{i_{2p}} ] \E[x_{j_1} x_{j_2}\cdots x_{j_{2q}}  ]       \big)\nonumber ,
	\end{align}
	where for each $m=1,2, \ldots, \min\{2p,2q\}$, $I_m$ is defined as
	\begin{equation}\label{def:I_k}
	I_{m}:=\{((i_1,\ldots,i_{2p}),(j_1,\ldots,j_{2q}))\in A_{2p}\times A_{2q}\suchthat |\{i_1,\ldots,i_{2p}\}\cap \{j_1,\ldots,j_{2q}\}|=m\}
	\end{equation}
and $|\{\cdot\}|$ denotes cardinality of the set  $\{\cdot\}$ (with multiplicity).
	
	Depending on the number of cross matching elements, there are two following cases:
	\vskip3pt
	\noindent \textbf{Case I.} \textbf{Odd number of cross matching:} 
	Suppose  $(i_1,i_2,\ldots,i_{2p})\in A_{2p}$ and $(j_1,j_2,\ldots,j_{2q})\in A_{2q}$, and $m=2k-1$, that is, $|\{i_1,i_2,\ldots,i_{2p}\}\cap \{j_1,j_2,\ldots,j_{2q}\}|=2k-1$ for $k=1,2,\ldots, \min\{p,q\}$.  
	 We shall show  that such $(i_1,i_2,\ldots,i_{2p})\in A_{2p}$ and $(j_1,j_2,\ldots,j_{2q})\in A_{2q}$ will have zero contribution in (\ref{eqn:I_k_RC}).
	
	First suppose $k=1$, that is, $|\{i_1,i_2,\ldots,i_{2p}\}\cap \{j_1,j_2,\ldots,j_{2q}\}|=1$. For $k=1$, the typical term of (\ref{eqn:I_k_RC}) will look like 
	$$ \sum_{I_{1}} \big( \E[x_{\ell_1}x_{i_2}\cdots x_{i_{2p}} x_{\ell_1} x_{j_2}\cdots x_{j_{2q}} ] -\E[ x_{\ell_1} x_{i_2}\cdots x_{i_{2p}} ] \E[x_{\ell_1} x_{j_2}\cdots x_{j_{2q}}  ]  \big),$$
	where $I_1 \subseteq A_{2p}\times A_{2q}$ and it looks like as
	$((\ell_1, i_2,\ldots,i_{2p}),(\ell_1,j_2,\ldots,j_{2q})).$
	
	 Since from (\ref{eqn:condition}), we have moment conditions on $x_i$, that is,
	\begin{equation*}
	\E[x_i]= 0, \ \E(x^2_i)=1 \mbox{ and } \sup_{i \geq 1}\E(|x_i|^{k})= \alpha_k < \infty \mbox{ for } k \geq 3.
	\end{equation*}  
	Therefore there exist $\gamma_1>0$, which depends only on $p, q$, such that 
	\begin{align} \label{eqn:alpha_1}
	\sum_{I_{1}} \big( \E[x_{\ell_1}x_{i_2}\cdots x_{i_{2p}} x_{\ell_1} x_{j_2}\cdots x_{j_{2q}} ] -\E[ x_{\ell_1} x_{i_2}\cdots x_{i_{2p}} ] \E[x_{\ell_1} x_{j_2}\cdots x_{j_{2q}}  ]  \big)
	& \leq \gamma_1 |I'_1|,
	\end{align}
	where  
	\begin{align*}
	I'_1 = \{((\ell_1, i_2,\ldots,i_{2p}),& (\ell_1,j_2,\ldots,j_{2q}))\in I_1 : \mbox{ each elements in } \{ \ell_1, i_2,\ldots,i_{2p}\} \cup \\ 
	& \{{\ell_1},{j_2}, \ldots, {j_{2q}} \}  
	 \mbox{ has multiplicity greater than or equal to two} \}.
	\end{align*}
 	Now we calculate the cardinality of $I'_1$. First recall the observation from Definition \ref{def:odd-even}, the maximum number of free variable in $(\ell_1, i_2,\ldots,i_{2p})$ will be $p$, when $(\ell_1, i_2,\ldots,i_{2p})$ is {\it odd-even pair matched}. Once we have chosen $(\ell_1, i_2,\ldots,i_{2p})$, the  entry $\ell_1$ of $({\ell_1},{j_2}, \ldots, {j_{2q}})$ is also chosen, therefore the maximum possible number of free variables from $({\ell_1},{j_2}, \ldots, {j_{2q}})$ will be $q-2$, where one extra $(-1)$ is coming due to the  constraint, $- \ell_1+\sum_{s=2}^{2q} (-1)^sj_s=0 \mbox{ (mod $n$) }$ and hence $|I'_1|= O(n^{p+q-2})$. Therefore from (\ref{eqn:alpha_1}),
	
	$$\frac{1}{n^{p+q-1}}\sum_{I_{1}} \big( \E[x_{i_1}\cdots x_{i_{2p}} x_{j_1} \cdots x_{j_{2q}} ] -\E[ x_{i_1}\cdots x_{i_{2p}} ] \E[x_{j_1}\cdots x_{j_{2q}}  ] \big) = o(1). $$


	Now for $k\geq 2$, 
 $I_{2k-1}$ looks like as
	$((\ell_1,\ldots,\ell_{2k-1},i_{2k}, \ldots,i_{2p}), (\ell_1,\ldots,\ell_{2k-1},j_{2k}, \ldots,j_{2q})).$
	
	As we have done for $k=1$, from (\ref{eqn:condition}), 
	there exist $\gamma_{2k-1}>0$, which depends only on $p, q$, such that the typical term of  (\ref{eqn:I_k_RC}) will be
	\begin{align} \label{eqn:alpha_2k-1}
	&\Big|\sum_{I_{2k-1}} \big( \E[x_{\ell_1}  \cdots x_{\ell_{2k-1}}  x_{i_{2k}} \cdots x_{i_{2p}} x_{\ell_1}\cdots x_{\ell_{2k-1}} x_{j_{2k}}\cdots x_{j_{2q}} ] \nonumber\\
	 & \qquad -\E[ x_{\ell_1}\cdots x_{\ell_{2k-1}}x_{i_{2k}} \cdots x_{i_{2p}} ] \E[x_{\ell_1}\cdots x_{\ell_{2k-1}} x_{j_{2k}}\cdots x_{j_{2q}} ]  \big)\Big| \nonumber \\
	& \leq \gamma_{2k-1} |I'_{2k-1}|,
	\end{align}
	where  
	\begin{align*}
	I'_{2k-1} = \{ ((\ell_1,\ldots, & \ell_{2k-1},i_{2k}, \ldots,i_{2p}), (\ell_1,\ldots,\ell_{2k-1},j_{2k}, \ldots,j_{2q}))\in I_{2k-1} : \mbox{ each  } \\
	& \mbox{ elements in }\{ \ell_1,\ldots,\ell_{2k-1},i_{2k}, \ldots,i_{2p}\} \cup \{\ell_1,\ldots,\ell_{2k-1},j_{2k}, \ldots,j_{2q} \} \\ 
	& \mbox{ has multiplicity greater than or equal to two} \}.
	\end{align*} 
	Now we calculate the cardinality of $I'_{2k-1}$. First observe that we shall get maximum number of free variables in $I'_{2k-1}$, if the following conditions hold:
	\begin{enumerate}
		\item [(i)] Each elements of $ \{\ell_1, \ell_2, \ldots, \ell_{2k-1}\}$ are distinct.
		\item [(ii)] 
		There exist $i^{*} \in (i_{2k},i_{2k+1},\ldots,i_{2p})$ such that $i^{*}$ equal to $\ell_s$ for $s=1,2, \ldots, 2k-1$. Without loss of generality, we suppose $i^{*}$ appears at even position and $\ell_s= \ell_{2k-1}$. That is, $ i^{*}= \ell_{2k-1}$. Similarly, there exist $j^{*} \in (j_{2k},j_{2k+1},\ldots,j_{2p})$ at odd position such that $j^{*}=\ell_{2k-2} $.
		\item [(iii)] $ \{\ell_1, \ell_2, \ldots, \ell_{2k-3}\} \cap \{i_{2k},\ldots,i_{2p}\} \setminus \{i^{*}\} \cap \{j_{2k},\ldots,j_{2p}\}\setminus \{j^{*}\} =\emptyset$.
		\item [(iv)] Each elements of $\{i_{2k},i_{2k+1},\ldots,i_{2p}\} \cup \{i^{*}\}$ and $\{j_{2k},j_{2k+1},\ldots,j_{2p}\} \cup \{j^{*}\}$ are {\it odd-even pair matched}.
	\end{enumerate}
	Note that, $(\ell_1,\ldots,\ell_{2k-1},i_{2k}, \ldots,i_{2p}) \in A_{2p} $ and $(\ell_1,\ldots,\ell_{2k-1},j_{2k}, \ldots,j_{2q}) \in A_{2q}$, that is, $$\sum_{r=1}^{2k-1}(-1)^r \ell_r + \sum_{r=2k}^{2p}(-1)^r i_r=0 \mbox{ (mod $n$) } \ \mbox{ and }\ 
	\sum_{r=1}^{2k-1}(-1)^r \ell_r + \sum_{r=2k}^{2q}(-1)^r j_r=0 \mbox{ (mod $n$)}.$$ But each elements of $\{i_{2k},i_{2k+1},\ldots,i_{2p}\} \cup \{\ell_{2k-1}\}$ and $\{j_{2k},j_{2k+1},\ldots,j_{2p}\} \cup \{\ell_{2k-2}\}$ are {\it odd-even pair matched}, therefore 
	the above two constraints will change into new constraints, namely, 
	\begin{equation} \label{eqn:2constraint L}
	\sum_{r=1}^{2k-2}(-1)^r \ell_r =0 \mbox{ (mod $n$)}\ \mbox{ and }\ \sum_{r=1}^{2k-3}(-1)^r \ell_r - \ell_{2k-1} =0 \mbox{ (mod $n$). }
	\end{equation}
	Hence the maximum number of free entries in $I'_{2k-1}$
	are $[(p-k+1)+(q-k+1) + (2k-3) + (-2)]= p+q-3$, where $(-2)$ is coming due to the constraints (\ref{eqn:2constraint L}). 
	
	Therefore from the above discussion on cardinality of $I'_{2k-1}$ and from (\ref{eqn:alpha_2k-1}),  we get
\begin{align*}
	&\frac{1}{n^{p+q-1}} \sum_{I_{2k-1}} \big( \E[x_{\ell_1}  \cdots x_{\ell_{2k-1}}  x_{i_{2k}} \cdots x_{i_{2p}} x_{\ell_1}\cdots x_{\ell_{2k-1}} x_{j_{2k}}\cdots x_{j_{2q}} ] \nonumber\\
	&\quad -\E[ x_{\ell_1}\cdots x_{\ell_{2k-1}}x_{i_{2k}} \cdots x_{i_{2p}} ] \E[x_{\ell_1}\cdots x_{\ell_{2k-1}} x_{j_{2k}}\cdots x_{j_{2q}} ]  \big) \nonumber \\
	& = o(1).
	\end{align*}
	Hence for each odd $m$, we get
	\begin{align} \label{eqn:I_odd}
	\lim_{n\to\infty}\frac{1} {n^{p+q-1}} & \sum_{I_{m}} \big( \E[x_{i_1}\cdots x_{i_{2p}} x_{j_1}\cdots x_{j_{2q}} ] -\E[ x_{i_1} \cdots x_{i_{2p}} ] \E[x_{j_1}\cdots x_{j_{2q}}  ]       \big) 
	 =0.
	\end{align}

\noindent \textbf{Case II.} \textbf{Even number of cross matching:} Suppose  $(i_1,i_2,\ldots,i_{2p})\in A_{2p}$ and $(j_1,j_2,\ldots,j_{2q})\in A_{2q}$, and $m=2k$, that is, $|\{i_1,i_2,\ldots,i_{2p}\}\cap \{j_1,j_2,\ldots,j_{2q}\}|=2k$ for some $k=1,2,\ldots, \min\{p,q\}$.  
  We shall show that such $(i_1,i_2,\ldots,i_{2p})\in A_{2p}$ and $(j_1,j_2,\ldots,j_{2q})\in A_{2q}$ will have non-zero contribution in (\ref{eqn:I_k_RC}). 

 In this case, a typical element of $I_{2k}$ looks like as $((\ell_1,\ldots,\ell_{2k},i_{2k+1}, \ldots,i_{2p}), (\ell_1,\ldots,\ell_{2k},j_{2k+1}, \ldots,j_{2q}))$.
Now we define $H \subseteq I_{2k}$ where $I_{2k}$ is as  in (\ref{def:I_k}).
 \begin{definition} \label{def:H}
 	$H$ is a subsets of all $((\ell_1,\ldots,\ell_{2k},i_{2k+1}, \ldots,i_{2p}), (\ell_1,\ldots,\ell_{2k},j_{2k+1}, \ldots,j_{2q}) ) \in I_{2k}$ such that 
 	\begin{enumerate}
 		\item [(i)] $ \{\ell_1, \ell_2, \ldots, \ell_{2k}\} \cap \{i_{2k+1},i_{2k+2},\ldots,i_{2p}\}\cap \{j_{2k+1},j_{2k+2},\ldots, j_{2q}\}=\emptyset$,
 		\item [(ii)] $(i_{2k},i_{2k+1},\ldots,i_{2p})$ and $(j_{2k},j_{2k+1},\ldots,j_{2p})$ are {\it odd-even pair matched}.
 	\end{enumerate}
 \end{definition}

By using the similar calculation, as we have done to calculate the number of free entries of $I_{2k-1}$, we shall get maximum number of free entries in $((\ell_1,\ldots,\ell_{2k},i_{2k+1}, \ldots,i_{2p}), (\ell_1,\ldots,\ell_{2k},j_{2k+1}, \ldots, j_{2q}) )$, when $((\ell_1,\ldots,\ell_{2k},i_{2k+1}, \ldots,i_{2p}), (\ell_1,\ldots,\ell_{2k},j_{2k+1}, \ldots,j_{2q}) ) \in H$. Now we calculate the cardinality of $H$.
 
 Suppose $((\ell_1,\ldots,\ell_{2k},i_{2k+1}, \ldots,i_{2p}), (\ell_1,\ldots,\ell_{2k},j_{2k+1}, \ldots,j_{2q}) ) \in H$. Since from (ii) condition of Definition \ref{def:H}, $(i_{2k+1},i_{2k+2},\ldots,i_{2p})$ and $(j_{2k+1},j_{2k+2},\ldots, j_{2q})$ are {\it odd-even pair matched}, therefore $\sum_{r=2k+1}^{2p}(-1)^r i_r=0 \mbox{ (mod $n$) }$ and $\sum_{r=2k+1}^{2q}(-1)^r j_r=0 \mbox{ (mod $n$)}$ are aromatically satisfy, and hence the constraints on  $(\ell_1,\ldots,\ell_{2k},i_{2k+1}, \ldots,i_{2p})$ and $(\ell_1,\ldots,\ell_{2k},j_{2k+1}, \ldots,j_{2q})$  will change into a single constraint, namely,
 \begin{equation} \label{eqn:constraint L}
 \sum_{r=1}^{2k}(-1)^r \ell_r =0 \mbox{ (mod $n$).}
 \end{equation}
 Now we first deal with $k\geq 2$, later we shall calculate for $k=1$.
	For $k=2,3, \ldots, \min\{p,q\}$, if we assume each elements of $ \{\ell_1, \ell_2, \ldots, \ell_{2k}\}$ are distinct, then cardinality of $H$ will be $O(n^{p+q-1})$, where $(-1)$ is coming due to the constraint $\sum_{r=1}^{2k}(-1)^r \ell_r =0 \mbox{ (mod $n$)}$. Note that in any other situation, like that any one of the conditions of Definition \ref{def:H} does not hold or elements of $ \{\ell_1, \ell_2, \ldots, \ell_{2k}\}$ are not distinct then cardinality of $H$ will be $o(n^{p+q-1})$.




%
 Since each elements of $ \{\ell_1, \ell_2, \ldots, \ell_{2k}\}$ are distinct, therefore 
 \begin{equation*} \label{eqn:H^c}
 \E[x_{\ell_1} \cdots x_{\ell_{2k}} x_{i_{2k+1}} \cdots x_{i_{2p}} ] \E[x_{\ell_1} \cdots x_{\ell_{2k}} x_{j_{2k+1}} \cdots x_{j_{2q}} ]=0.
 \end{equation*}
 Hence for each $k\geq 2$, we get
  \begin{align} \label{eqn:I_even}
  &\lim_{n\to\infty}\frac{1} {n^{p+q-1}}  \sum_{I_{2k}} \big( \E[x_{i_1}x_{i_2}\cdots x_{i_{2p}} x_{j_1} x_{j_2}\cdots x_{j_{2q}} ] -\E[ x_{i_1} x_{i_2}\cdots x_{i_{2p}} ] \E[x_{j_1} x_{j_2}\cdots x_{j_{2q}}  ]       \big) \nonumber \\
  & = \lim_{n\to\infty}\frac{1} {n^{p+q-1}}  \sum_{I_{2k}} \E[x_{i_1}x_{i_2}\cdots x_{i_{2p}} x_{j_1} x_{j_2}\cdots x_{j_{2q}} ]
  =O(1).
  \end{align}
 Now  combining both the Cases, (\ref{eqn:I_odd}) and (\ref{eqn:I_even}), from (\ref{eqn:I_k_RC}) we get
   \begin{align} \label{eqn:I_2k2_RC}
  \lim_{n\to\infty} \Cov\big(w_p,w_q\big) 
  =& \lim_{n\to\infty}\frac{1}{n^{p+q-1}}\sum_{k=1}^{\min\{p,q\}} \sum_{I_{2k}} \big( \E[x_{i_1}x_{i_2}\cdots x_{i_{2p}} x_{j_1} x_{j_2}\cdots x_{j_{2q}} ]    \big).
  \end{align} 
 Now for  fixed $k=2, 3, \ldots, \min\{p,q\}$, we calculate 
 the contribution of each term of (\ref{eqn:I_2k2_RC}).
 \begin{align} \label{eqn:I_2k_c_k_RC}
 & \lim_{n\to\infty}  \frac{1}{n^{p+q-1}}\sum_{I_{2k}} \E[x_{i_1}x_{i_2}\cdots x_{i_{2p}} x_{j_1} x_{j_2}\cdots x_{j_{2q}} ] \nonumber\\
 & = \lim_{n\to\infty}\frac{1}{n^{p+q-1}}\sum_{A'_{2k}} \sum_{A'_{2k}} c_k n^{(p-k) + (q-k)} \E[x_{i_1}x_{i_2}\cdots x_{i_{2k}} x_{j_1} x_{j_2}\cdots x_{j_{2k}} ] \nonumber\\
 & = c_k \lim_{n\to\infty}\frac{1}{n^{2k-1}}\sum_{A'_{2k}} \sum_{A'_{2k}} \E[x_{i_1}x_{i_2}\cdots x_{i_{2k}} x_{j_1} x_{j_2}\cdots x_{j_{2k}} ],
 \end{align}
 where $c_k=\binom{p}{p-k}^2(p-k)! \binom{q}{q-k}^2(q-k)!.$ For the odd-even pair matching among $(2p-2k)$ variables from $(i_1,\ldots, i_{2p})$, first we choose $(p-k)$ odd and $(p-k)$ even position from the available $p$ odd and $p$ even position in $\binom{p}{p-k}^2$ ways. After choosing $(p-k)$ odd, $(p-k)$ even positions and $(p-k)$ free variables in odd positions, the random variables in even positions can permute among themselves (satisfying the condition of odd-even pair matching) in  $(p-k)!$  ways. Hence, odd-even pair matching among $(2p-2k)$ variables of $(i_1,\ldots, i_{2p})$  happens in  $\binom{p}{p-k}^2(p-k)!n^{p-k}$ ways. Similarly, odd-even pair matching happens among $(2q-2k)$ variables from $(j_1,j_2,\ldots, j_{2q})$ in $\binom{q}{q-k}^2(q-k)!n^{q-k}$ ways. 
 The rest of the variable $\{i_1,i_2,\ldots,i_{2k}\}$ and $\{j_1,j_2,\ldots,j_{2k}\}$ will (cross) match completely and both belong to $A_{2k}'$. Now from (\ref{eqn:I_2k_c_k_RC}), we get 
 \begin{align}\label{reduced from 2}
 \lim_{n\to \infty} \frac{1}{n^{p+q-1}}  \sum_{I_{2k}}\E[x_{i_1}\ldots x_{i_{2p}}   x_{j_1}\ldots x_{j_{2q}}] 
 & =c_k\lim_{n\to \infty}\frac{1}{n^{2k-1}}\sum_{A_{2k}'}\sum_{A_{2k}'}\E[x_{i_1}\ldots x_{i_{2k}}x_{j_1}\ldots x_{j_{2k}}]\nonumber \\
 &=c_k\lim_{n\to \infty}\frac{1}{n^{2k-1}}\sum_{s,t=-(k-1)}^{k-1}\sum_{A_{2k,s}'}\sum_{A_{2k,t}'}\E[x_{i_1}\ldots x_{i_{2k}}x_{j_1}\ldots x_{j_{2k}}]\nonumber \\
 &=c_k \sum_{s=-(k-1)}^{k-1}\lim_{n\to \infty}\frac{|A_{2k,s}'| (2-{\bf 1}_{\{s=0\}})k!k!}{n^{2k-1}}.
 \end{align}
 The constant $(2-{\bf 1}_{\{s=0\}})$ appeared because for $s\neq 0$, $A_{2k,s}'$ can (cross) match completely  with $A_{2k,s}'$ and $A_{2k,-s}'$ . 
 The factor $k!^2$ appeared because for a fixed choice of $(i_1,\ldots,i_{2k})\in A_{2k,s}'$,  we can choose the same set of values  from $\{j_1,\ldots, j_{2k}\}$ in $k!k!$ different ways, permuting the odd positions and the even positions among themselves. 
 
Since from Result \ref{result:cardinalityA}, we have 
 \begin{align}\label{3}
 \lim_{n\to \infty}\frac{|A_{2k,s}'|}{n^{2k-1}}=\lim_{n\to \infty}\frac{|A_{2k,s}|}{n^{2k-1}}
 &=\frac{1}{(2k-1)!}\sum_{j=0}^{k+s-1}(-1)^j\binom{2k}{j}(k+s-j)^{2k-1}.
 \end{align}
 Therefore for each fixed $k\geq 2$, from (\ref{reduced from 2}) and (\ref{3}), we get
 
\begin{align}\label{reduced from 3}
 \lim_{n\to \infty} \frac{1}{n^{p+q-1}}  \sum_{I_{2k}}\E[x_{i_1}\ldots x_{i_{2p}}   x_{j_1}\ldots x_{j_{2q}}] 
 & = c_k \sum_{s=-(k-1)}^{k-1} \sum_{j=0}^{k+s-1}(-1)^j\binom{2k}{j}(k+s-j)^{2k-1}(2-{\bf 1}_{\{s=0\}})k!k! \nonumber \\
 &= c_k g(k), \mbox{ say}.
 \end{align} 
 
 Now  we are left with $k=1$ in Case II. If $k=1$ then from \eqref{eqn:constraint L}, we get  $\ell_1=\ell_2$ and therefore 
 \begin{align}\label{4}
 \lim_{n\to \infty}\frac{1}{n^{2p-1}}\sum_{I_{2}}\l(\E[x_{i_1}\ldots x_{i_{2p}}x_{j_1}\ldots x_{j_{2q}}]-\E[x_{i_1}\ldots x_{i_{2p}}]\E[x_{j_1}\ldots x_{j_{2q}}]\r) & =c_1(\E x_1^4-(\E x_1^2)^2) \nonumber\\
 &=(\E x_1^4-1)c_1,
 \end{align}
 where  $c_1=\binom{p}{p-1}^2(p-1)! \binom{q}{q-1}^2(q-1)!.$
 Hence from (\ref{eqn:I_2k2_RC}), \eqref{reduced from 3} and \eqref{4}, we get 
 $$
 \lim_{n\to\infty} \Cov\big(w_p,w_q\big)=\sum_{k=2}^{\min\{p,q\}} c_kg(k)+(\E x_1^4-1)c_1,
 $$
 where  $c_k=\binom{p}{p-k}^2(p-k)! \binom{q}{q-k}^2(q-k)!$ and $g(k)$ is given by
 $$
 g(k)=\frac{1}{(2k-1)!}\sum_{s=-(k-1)}^{k-1}\sum_{j=0}^{k+s-1}(-1)^j\binom{2k}{j}(k+s-j)^{2k-1}(2-{\bf 1}_{\{s=0\}})k!k!.
 $$
This complete the proof of the theorem \ref{thm:revcircovar}.
\end{proof}

\section{Proof of Theorem \ref{thm:revcirpoly}}\label{sec:poly}

We begin with some notation and definitions which will be used in the proof. First recall $A_{2p}$ from \eqref{def:A_2p} in Section \ref{sec:poly},
$$A_{2p}=\{(i_1,\ldots,i_{2p})\in \mathbb N^{2p}\suchthat \sum_{k=1}^{2p}(-1)^ki_k=0 \mbox{ (mod $n$) }, 1\le i_1,\ldots,i_{2p}\le n\}.$$ 
For a  vector  $ J = (j_1, j_2, \ldots, j_{2p})\in A_{2p},$ we define a multi-set $S_J$ as 
\begin{equation}\label{def:S_j}
S_{J} = \{j_1, j_2, \ldots,j_{2p}\}.
\end{equation}
\begin{definition}\label{def:connected}
	Two vectors $J =(j_1, j_2, \ldots, j_{2p})$ and $J' = (j'_1, j'_2, \ldots, j'_{2p})$, where $J \in A_{2p}$ and $J' \in A_{2q}$, are said to be \textit{connected} if
	$S_{J}\cap S_{J'} \neq \emptyset$.	
\end{definition}
For $1 \leq i \leq \ell$, suppose $J_i \in A_{2p_i}$. Now, we define cross-matched and self-matched element in $\displaystyle{\cup_{i=1}^{\ell} S_{J_i} }$.
\begin{definition}\label{def:cross matched}
	An element in $\displaystyle{\cup_{i=1}^{\ell} S_{J_i} }$ is called \textit{cross-matched} if it appears at least in two distinct $S_{J_i}$. If it appears in $k$ many ${S_{J_i}}^,s$, then we say its \textit{cross-multiplicity} is $k$.
\end{definition}
\begin{definition}\label{def:self-matched}
	An element in $\displaystyle{\cup_{i=1}^{\ell} S_{J_i} }$ is called \textit{self-matched} if it appears more than once in one of $S_{J_i}$. If it appears $k$ many times in $S_{J_i}$, then we say its \textit{self-multiplicity} in $S_{J_i}$ is $k$.
	
	An element in $\displaystyle{\cup_{i=1}^{\ell} S_{J_i} }$  can be both self-matched and cross-matched. If an element of $\displaystyle{\cup_{i=1}^{\ell} S_{J_i} }$ has \textit{cross-multiplicity} one, that means, it appears only in one of the ${S_{J_i}}$. 
\end{definition}	
\begin{definition}\label{def:cluster}
	Given a set of vectors $S= \{J_1, J_2, \ldots, J_\ell \}$, where $J_i \in A_{2p_i}$ for $1 \leq i \leq \ell$, a subset $T=\{J_{n_1}, J_{n_2}, \ldots, J_{n_k}\}$ of $S$ is called a \textit{cluster} if it satisfies the following two conditions: 
	\begin{enumerate}
		\item[(i)] For any pair $J_{n_i}, J_{n_j}$ from  $T$ one can find a chain of vectors from 
		$T$, which starts with $J_{n_i}$ and ends with $J_{n_j}$ such that any two neighbouring vectors in the chain are connected.
		\item[(ii)] The subset $\{J_{n_1}, J_{n_2}, \ldots, J_{n_k}\}$ can not  be enlarged to a subset which preserves condition (i).
	\end{enumerate}
\end{definition}
For more details about cluster, we refer the readers to \cite{maurya2019process}, where they have explained the structure of cluster by using graph.

Now we define $B_{P_\ell} \subseteq   A_{2p_1} \times A_{2p_2} \times \cdots \times A_{2p_\ell}$, where $A_{2p_i}$ is as defined in \eqref{def:A_2p}.  
\begin{definition}\label{def:B_{P_l}}
	Let $\ell \geq 2$ and  $P_\ell = (2p_1,2p_2, \ldots, 2p_\ell ) $. Now $ B_{P_\ell}$ is a subset of $ A_{2p_1} \times A_{2p_2} \times \cdots \times A_{2p_\ell}$ such that  $ (J_1, J_2, \ldots, J_\ell) \in B_{P_\ell} $ if 
	\begin{enumerate} 
		\item[(i)] $\{J_1, J_2, \ldots, J_\ell\} $ form a cluster, 
		\item[(ii)] each element in  $\displaystyle{\cup_{i=1}^{\ell} S_{J_i} }$ has  multiplicity greater than or equal to two. 
	\end{enumerate}
\end{definition} 
 The next lemma gives us the cardinality of $B_{P_\ell}$.
\begin{lemma}\label{lem:cluster}
	For  $\ell \geq 3 $, 
	\begin{equation}\label{equation:cluster}
	|B_{P_\ell }| = o \big(n^{p_1+p_2 + \cdots + p_\ell -\frac{\ell}{2} }\big).
	\end{equation}
\end{lemma}
\begin{remark}
	The above lemma is not true if $\ell=2$ and $p_1= p_2$. Suppose $(J_1,J_2)\in B_{P_2}$. Then all $2p_1$ entries of $J_1$ may coincides with $2p_2(=2p_1)$ many entries of $J_2$ and hence 
	$$|B_{P_2}|=O(n^{2p_1-1}).$$
	So in this situation, $|B_{P_2}|>o(n^{p_1+p_2-1})$.
\end{remark}
The above lemma  is similar to Lemma 15 of \cite{maurya2019process}, where the authors have proved a similar result  with a different set of constraints on the sets $A_{2p_1},\ldots,A_{2p_k}$. To prove Lemma \ref{lem:cluster}, we shall use similar ideas. We prove Lemma \ref{lem:cluster} after the proof of Theorem \ref{thm:revcirpoly}. The following lemma is an easy consequence of Lemma \ref{lem:cluster}.

\begin{lemma}\label{lem:maincluster}
	Suppose $\{J_1, J_2, \ldots, J_\ell \} $ form a cluster where $J_i\in A_{2p_i}$ with $p_i\geq 1$ for $1\leq i\leq \ell$, $\{x_i\}_{i \geq 1}$ is independent and satisfies (\ref{eqn:condition}). Then for $\ell \geq 3,$
	\begin{equation}\label{equation:maincluster}
	\frac{1}{ n^{p_1+p_2+ \cdots + p_\ell - \frac{\ell}{2}}} \sum_{A_{2p_1}, A_{2p_2}, \ldots, A_{2p_\ell}} \E\Big[\prod_{k=1}^{\ell}\Big(x_{J_k} - \E(x_{J_k})\Big)\Big] = o(1),
	\end{equation}
	where 
	$$J_k = (j^{k}_{1}, j^{k}_{2}, \ldots, j^{k}_{2p_k} ) \ \mbox{and} \ x_{J_k} = x_{j^{k}_{1}} x_{j^{k}_{2}} \cdots x_{j^{k}_{2p_k}}.$$	
\end{lemma}
\begin{proof} First observe that   $\E\Big[\prod_{k=1}^{\ell}\Big(x_{J_k} - \E(x_{J_k})\Big)\Big]$ will be non-zero only if each $x_i$ appears at least twice in the collection $\{ x_{j^{k}_{1}}, x_{j^{k}_{2}}, \ldots ,x_{j^{k}_{2p_k}} ; 1\leq k\leq \ell\}$, because $\E(x_i)=0$ for each $i$. Therefore
	\begin{equation}\label{eqn:equality_reduction}
	\sum_{A_{2p_1}, \ldots, A_{2p_\ell}} \hspace{-3pt}\E\Big[\prod_{k=1}^{\ell}\Big(x_{J_k} - \E(x_{J_k})\Big)\Big]=\sum_{(J_1,\ldots,J_\ell)\in B_{P_\ell}} \hspace{-3pt} \E\Big[\prod_{k=1}^{\ell}\Big(x_{J_k} - \E(x_{J_k})\Big)\Big],
	\end{equation} 
	where $B_{P_\ell}$ as in Definition \ref{def:B_{P_l}}. Since from (\ref{eqn:condition}), we have  
	\begin{equation*}\label{eqn:higher moment finite}
	\E(x^2_i)=1 \mbox{ and } \sup_{i \geq 1}\E(|x_i|^{k})= \alpha_k < \infty \mbox{ for } k \geq 3.
	\end{equation*}  
	 Therefore for $p_1,p_2,\ldots,p_\ell \geq 1$, there exists $\beta_\ell>0$, which depends only on $p_1,p_2,\ldots,p_\ell$, such that  
	\begin{equation}\label{eqn:modulus finite}
	\Big|\E\big[\prod_{k=1}^{\ell}\big(x_{J_k} - \E(x_{J_k})\big)\big]\Big|\leq \beta_\ell
	\end{equation} 
	for all $(J_1, J_2, \ldots, J_\ell)\in A_{2p_1}\times A_{2p_2}\times \cdots \times A_{2p_\ell}$.
	
	Now using \eqref{eqn:equality_reduction} and \eqref{eqn:modulus finite}, we have
	\begin{align*}
	\sum_{A_{2p_1}, A_{2p_2}, \ldots, A_{2p_\ell}} \Big|\E\big[\prod_{k=1}^{\ell}\big(x_{J_k} - \E(x_{J_k})\big)\big]\Big|
	& \leq  \sum_{(J_1,J_2,\ldots,J_\ell)\in B_{P_\ell}} \beta_{\ell} 
	\ = |B_{p_\ell}| \ \beta_\ell.
	\end{align*}
	By using Lemma \ref{lem:cluster} in above expression, we get
	\begin{align*}
	\sum_{A_{2p_1}, A_{2p_2}, \ldots, A_{2p_\ell}} \Big|\E\big[\prod_{k=1}^{\ell}\big(x_{J_k} - \E(x_{J_k})\big)\big]\Big| = o \big(n^{p_1+p_2 + \cdots + p_\ell -\frac{\ell}{2} }\big),
	\end{align*}
	and hence 
	\begin{equation*}
	\frac{1}{ n^{p_1+p_2+ \cdots + p_\ell - \frac{\ell}{2}}} \sum_{A_{2p_1}, A_{2p_2}, \ldots, A_{2p_\ell}} \E\Big[\prod_{k=1}^{\ell}\Big(x_{J_k} - \E(x_{J_k})\Big)\Big] = o(1).
	\end{equation*}
	
	This completes the proof of lemma. 
\end{proof}

We shall use the above lemmata and Theorem \ref{thm:revcircovar} to prove Theorem \ref{thm:revcirpoly}. 

\begin{proof}[Proof of Theorem \ref{thm:revcirpoly}] We use method of moments and  Wick's formula to prove Theorem \ref{thm:revcirpoly}. First recall that from the method of moments, to prove $w_Q \stackrel{d}{\longrightarrow} N(0,\sigma_{Q}^2)$, it is sufficient to show that
	\begin{align} \label{eqn:moment w_Q}
	\lim_{n\to\infty} \E[ (w_Q)^\ell] = \E[ (N(0, \sigma^2_{Q}))^\ell ] \ \ \forall \ \ell=1,2, \ldots.
	\end{align}
	So, to prove (\ref{eqn:moment w_Q}), it is enough to show that, for $p_1, p_2, \ldots , p_\ell \geq 1$,
	$$\lim_{n\to\infty}\E[w_{p_1}w_{p_2} \cdots w_{p_\ell}]=\E[N_{p_1}N_{p_2} \cdots N_{p_\ell}],$$
	where $\{N_{p}\}_{p \geq 1}$ is a centered Gaussian family with covariance $\sigma_{p,q}$, as in (\ref{eqn:sigma_p,q}).
	Now using trace formula (\ref{trace formula RC_n}), we have  
	\begin{align*}
	w_{p_k}  = \frac{1}{\sqrt{n}} \Big(\Tr(RC_n)^{2p_k} - \E[\Tr(RC_n)^{2p_k}]\Big)
	= \frac{1}{n^{p_k -\frac{1}{2}}} \sum_{A_{2p_k}} \Big( x_{j^{k}_{1}}\cdots x_{j^{k}_{2p_k}} - \E[x_{j^{k}_{1}}\cdots x_{j^{k}_{2p_k}}]\Big).
	\end{align*}
	Note that in the above summation $(j^{k}_{1}, j^{k}_{2},\ldots, j^{k}_{2p_k})\in A_{2p_k}$.
	Therefore 
	\begin{align}\label{eqn:expectation_thm2}
	&\quad \E[w_{p_1}w_{p_2} \cdots w_{p_\ell}] \\
	&= \frac{1}{n^{p_1 + p_2 + \cdots +p_\ell -\frac{\ell}{2}}} \sum_{A_{2p_1}, A_{2p_2}, \ldots, A_{2p_\ell}} \E\big[ (x_{J_1} - \E x_{J_1}) (x_{J_2} - \E x_{J_2})  \cdots (x_{J_\ell} - \E x_{J_\ell})\big].\nonumber
	\end{align}
	Now for a fixed $J_1,J_2,\ldots,J_\ell$, if there exists a $k\in\{1,2,\ldots,\ell\}$ such that $J_k$ is not connected with any $J_i$ for $i\neq k$, then 
	$$\E\big[ (x_{J_1} - \E x_{J_1}) (x_{J_2} - \E x_{J_2})  \cdots (x_{J_\ell} - \E x_{J_\ell})\big]=0$$
	due to the independence of $\{x_i\}_{i\geq 1}$.
	
	Therefore $J_1,J_2,\ldots,J_\ell$ must form  clusters with each cluster length greater than or equal to two, that is, each cluster should contain at least two vectors. Suppose $G_1,G_2,\ldots,G_s$ are the clusters formed by  vectors $J_1,J_2,\ldots,J_\ell$ and   $|G_i|\geq 2$ for all $1\leq i \leq s$ where $|G_i|$ denotes the length of the cluster $G_i$. Observe that $\sum_{i=1}^s |G_i|=\ell$.   
	
	If there exists  a cluster $G_j$ among $G_1,G_2,\ldots,G_s$ such that $|G_j|\geq 3$, then from Theorem \ref{thm:revcircovar} and Lemma \ref{lem:maincluster}, we have  
	\begin{align*}
	\frac{1}{n^{p_1 + p_2 + \cdots +p_\ell -\frac{\ell}{2}}} \sum_{A_{2p_1}, A_{2p_2}, \ldots, A_{2p_\ell}} \E\big[ (x_{J_1} - \E x_{J_1}) (x_{J_2} - \E x_{J_2})  \cdots (x_{J_\ell} - \E x_{J_\ell})\big]=o(1).
	\end{align*}  
	Therefore, if $\ell$ is odd then there will be a cluster of odd length and hence 
	$$ \lim_{n\tends \infty}  \E[w_{p_1} w_{p_2}\cdots w_{p_\ell}] = 0.$$ 
	
	Similarly, if $\ell$ is even then the contribution due to $\{ J_1, J_2, \ldots, J_\ell \}$ to  $ \E[w_{p_1} w_{p_2}\cdots w_{p_\ell}]$ is $O(1)$ only when $\{ J_1, J_2, \ldots, J_\ell\}$  decomposes into clusters of length 2. Therefore from \eqref{eqn:expectation_thm2}, we get
	\begin{align*} \label{eq:multisplit}
	& \quad \lim_{n\tends \infty}  \E[w_{p_1} w_{p_2}\cdots w_{p_\ell}]\\
	& =\lim_{n\to\infty} \frac{1}{n^{p_1 + p_2 + \cdots +p_\ell -\frac{\ell}{2}}} \sum_{A_{2p_1}, A_{2p_2}, \ldots, A_{2p_\ell}} \E\big[ (x_{J_1} - \E x_{J_1}) (x_{J_2} - \E x_{J_2})  \cdots (x_{J_\ell} - \E x_{J_\ell})\big]\\
	& =\lim_{n\to\infty} \frac{1}{n^{p_1 + p_2 + \cdots +p_\ell -\frac{\ell}{2}}} \sum_{\pi \in \mathcal P_2(\ell)} \prod_{i=1}^{\frac{\ell}{2}}  \sum_{A_{2p_{y(i)}},\ A_{2p_{z(i)}}} \E\big[ (x_{J_{y(i)}} - \E x_{J_{y(i)}}) (x_{J_{z(i)}} - \E x_{J_{z(i)}})\big],
	\end{align*}
	where  $\pi = \big\{ \{y(1), z(1) \}, \ldots , \{y(\frac{\ell}{2}), z(\frac{\ell}{2})  \} \big\}\in \mathcal P_2(\ell)$ and $\mathcal P_2(\ell)$ is the set of all pair partition of $ \{1, 2, \ldots, \ell\} $. Using Theorem \ref{thm:revcircovar}, from the last equation we get
	\begin{equation}\label{eqn:product of expectation}
	\lim_{n\tends \infty}  \E[w_{p_1}w_{p_2} \cdots w_{p_\ell}]
	=\sum_{\pi \in P_2(\ell)} \prod_{i=1}^{\frac{\ell}{2}} \lim_{n\tends \infty} \E[w_{p_{y(i)}} w_{p_{z(i)}}].
	\end{equation}
    Since from Theorem \ref{thm:revcircovar}, we have 
	$$\lim_{n\to\infty}\E(w_p w_q) = \sigma_{p,q} = \E(N_p N_q).$$
	Therefore using Wick's formula,  from \eqref{eqn:product of expectation} we get 
	\begin{align*}
	\lim_{n\tends \infty}  \E[w_{p_1}w_{p_2} \cdots w_{p_\ell}]
	& =\sum_{\pi \in P_2(\ell)} \prod_{i=1}^{\frac{\ell}{2}} \lim_{n\tends \infty} \E[w_{p_{y(i)}} w_{p_{z(i)}}]\\
	& =\sum_{\pi \in \mathcal P_2(\ell)} \prod_{i=1}^{\frac{\ell}{2}} \E[N_{p_{y(i)}} N_{p_{z(i)}} ] \\
	&=\E[ N_{p_1}N_{p_2} \cdots N_{p_\ell} ].
	\end{align*}
	This completes the proof of Theorem \ref{thm:revcirpoly}.
\end{proof}
Now we prove Lemma \ref{lem:cluster}.
\begin{proof}[Proof of Lemma \ref{lem:cluster}]
	We define  subsets $B^{'}_{P_\ell}$ and $B^{''}_{P_\ell}$ of $B_{P_\ell}$ as
	\begin{align}\label{def:B'}
	B^{'}_{P_\ell} =  \{ (J_1, J_2, \ldots, J_\ell) \in B_{P_\ell}  :& \mbox{ \textit{cross-multiplicity} of each entry of }\cup_{i=1}^{\ell} S_{J_i}  \\
	&	 \mbox{ is less than or equal to two} \}\nonumber
	\end{align} 
	and
	\begin{equation}\label{def:B''}
	B^{''}_{P_\ell} =  B_{P_\ell} \setminus B^{'}_{P_\ell}.
	\end{equation}
	Observe that $$ |B_{P_\ell}| = |B^{'}_{P_\ell}|+ |B^{''}_{P_\ell}|.$$
	Since in $B^{''}_{P_\ell}$ there exists at least one entry with \textit{cross-multiplicity} greater than or equal to three, the maximum number of free entries (which can be chosen freely from $1$ to $n$) in $B^{''}_{P_\ell}$ will be less than or equal to the maximum number of free entries in $B^{'}_{P_\ell}.$ Therefore 
	$$|B^{''}_{P_\ell}| \leq |B^{'}_{P_\ell}|$$
	and hence
	\begin{equation} \label{eqn:B_P,B_P'}
	|B_{P_\ell}| = O(|B^{'}_{P_\ell}|).
	\end{equation}
	So to prove (\ref{equation:cluster}), it is enough to prove that
	\begin{equation}\label{equation:cluster1}
	|B^{'}_{P_\ell }| = o \big(n^{p_1+p_2 + \cdots + p_\ell -\frac{\ell}{2} }\big).
	\end{equation}
	Now we shall prove (\ref{equation:cluster1}) by using mathematical induction on $\ell$. Let us first prove it for $\ell=3$. Suppose $(J_1, J_2, J_3)\in B_{P_3}$ with $J_i = (j^{i}_1, j^{i}_2, \ldots, j^{i}_{2p_i})$ for $i=1, 2, 3,$ then $(J_1, J_2, J_3)$ form a cluster and they can be connected in the following way:\\
	Suppose $r$ many entries of $J_1$ matches with  $r$ many entries of $J_2$ and $J_3$ both; $r_1$ many entries of $J_1$ matches with $r_1$ many entries from $J_2$ only; $r_2$ many entries of $J_2$ matches with $r_2$ many entries of $J_3$ only; $r_3$ many entries of $J_3$ matches with $r_3$ many entries of $J_1$ only, where
	\begin{equation}\label{eqn:r_3}
	r, r_1, r_2, r_3 \geq  0 \ , \ r+r_1 + r_2 \leq  2p_2,\ r+r_1 + r_3\leq 2p_1, \ r+r_2+r_3\leq 2p_3. 
	\end{equation}  
	\begin{figure}[h]
		\centering\vskip-10pt
		\includegraphics[height=40mm, width =50mm ]{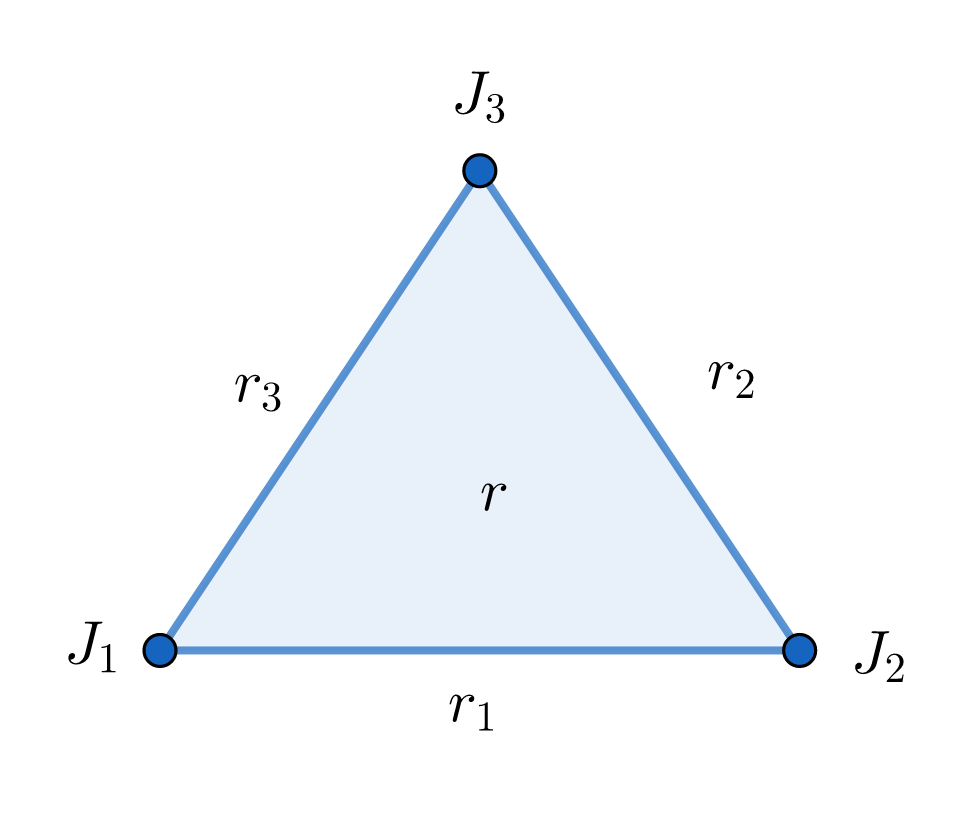}
		\vskip-10pt
		\caption{Connection between $J_1,J_2$ and $J_3$.}	 
	\end{figure}  
	Note that if $r =0$, that is, all entries of $\displaystyle{\cup_{i=1}^{3} S_{J_i} }$ has \textit{cross-multiplicity} less than or equal to two, then $(J_1, J_2, J_3)$ belongs to $B^{'}_{P_3}$ as defined in (\ref{def:B'}). Also if $r\neq 0$, that is, there are $r$ many entries of $\displaystyle{\cup_{i=1}^{3} S_{J_i} }$ which has \textit{cross-multiplicity} three, then $(J_1, J_2, J_3)$ belongs to $B^{''}_{P_3}$ as defined in (\ref{def:B''}). Now we calculate cardinality of $B^{'}_{P_3}$, that is, the case $r =0$, to prove (\ref{equation:cluster}). 
	
	Recall from the Case I and Case II of proof of Theorem \ref{thm:revcircovar}, the number of free entries between two vectors which have same $k$ many entries, will be maximum when $k$ is even. Therefore to obtain maximum number of free entries from $B^{'}_{P_3}$, we shall assume each $r_i$ is even, for $i=1,2,3.$ Otherwise, we shall get less contribution.
	
	 Since $(J_1, J_2, J_3)$ is forming a cluster, therefore at least two $r_i$ have to be non-zero. Then the following two cases arise.\\
	\noindent \textbf{Case I.} Exactly two $r_i$ are non-zero and one $r_i$ is zero. Without loss of generality, we assume $r_3 = 0.$ Therefore we can represent $J_i$ as
	\begin{align*}
	J_1 & = (x^{1}_1, x^{1}_2, \ldots, x^{1}_{r_1}, j^{1}_{r_1+1}, j^{1}_{r_1+2}, \ldots, j^{1}_{2p_1}), \\
	J_2 &= (x^{1}_1, x^{1}_2, \ldots, x^{1}_{r_1}, x^{2}_1, x^{2}_2, \ldots, x^{2}_{r_2}, j^{2}_{r_1+r_2+1}, j^{2}_{r_1+ r_2+2}, \ldots, j^{2}_{2p_2}), \\
	J_3 & = (x^{2}_1, x^{2}_2, \ldots, x^{2}_{r_2}, j^{3}_{r_2+1}, j^{3}_{r_2+2}, \ldots, j^{3}_{2p_3}),
	\end{align*}
	where (\ref{eqn:r_3}) will be of the following form 
	$$   r_1 \leq 2p_1, \ r_1 + r_2 \leq  2p_2, \ r_2\leq 2p_3$$
	and  Figure $1$ will look like as Figure $2$.
	\begin{figure}[h]
		\centering\vskip-5pt
		\includegraphics[height=20mm, width =70mm ]{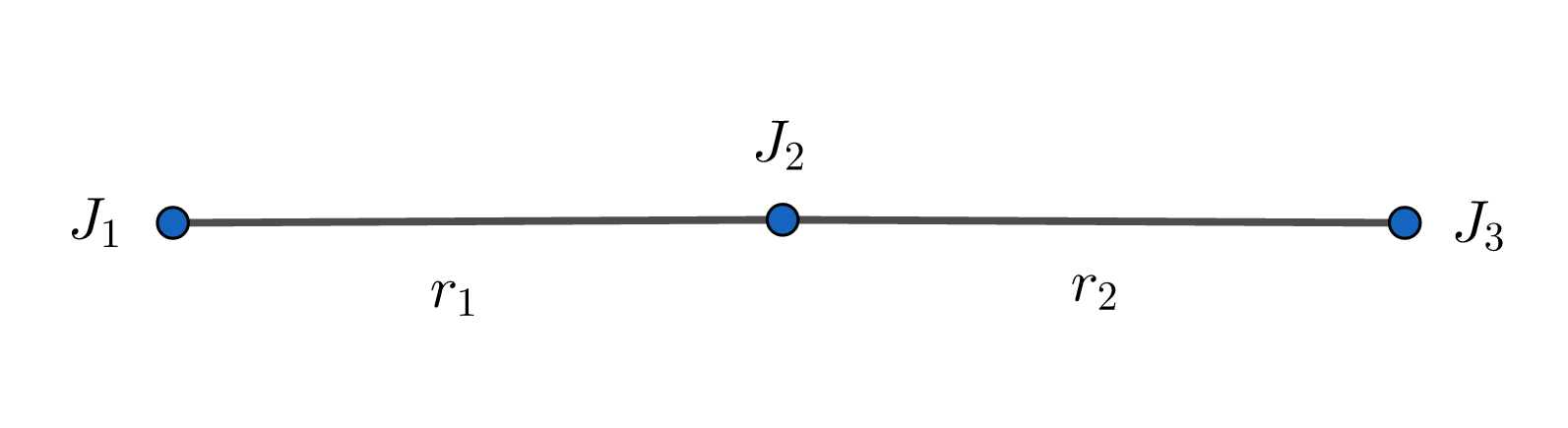}
		\vskip-10pt
		\caption{Connection between $J_1,J_2$ and $J_3$: Case I.}	
	\end{figure} 
As each entry in $\displaystyle{\cup_{i=1}^{3} S_{J_i} }$ has multiplicity at least two, then we shall get maximum contribution from $J_1$, if the following holds:
	\begin{enumerate}
		\item [(i)] each elements of $\{x^{1}_1, x^{1}_2, \ldots, x^{1}_{r_1} \}$ has to be distinct,
		\item [(ii)] $\{x^{1}_1, x^{1}_2, \ldots, x^{1}_{r_1} \}$ have no common element with $\{j^{1}_{r_1+1}, j^{1}_{r_1+2}, \ldots, j^{1}_{2p_1}\}$,
		\item [(iii)] $(j^{1}_{r_1+1}, j^{1}_{r_1+2}, \ldots, j^{1}_{2p_1})$ is {\it odd-even pair matched}. 
	\end{enumerate}
    Since $(j^{1}_{r_1+1}, j^{1}_{r_1+2}, \ldots, j^{1}_{2p_1})$ is {\it odd-even pair matched}, that is, $\sum_{s=r_1+1}^{2p_1} (-1)^s j^{1}_s = 0 \;{(\mbox{mod $n$})}$, therefore the constraint on the entries of $J_1$, that is, $\sum_{s=1}^{r_1} (-1)^s x^{1}_s + \sum_{s=r_1+1}^{p_1} (-1)^s j^{1}_s = 0 \;{(\mbox{mod $n$})}$ will change automatically into $\sum_{s=1}^{r_1} (-1)^s x^{1}_s = 0 \;{(\mbox{mod $n$})}$. So the maximum contribution from $J_{1}$ is $O\big(n^{(r_1+ \frac{2p_1-r_1}{2}-1)}\big)$, where $(-1)$ is coming due to the constraint $\sum_{s=1}^{r_1} (-1)^s x^{1}_s = 0 \;{(\mbox{mod $n$})}$.
	Now after fixing entries in $J_1$, by similar argument, the maximum contribution from $J_3$ will be $O\big(n^{(r_2 + \frac{2p_3-r_2}{2}-1)}\big)$, where $(-1)$ is coming due to the constraint $\sum_{s=1}^{r_2} (-1)^s x^{2}_s = 0 \;{(\mbox{mod $n$})}$. After selecting $J_1, J_3;$ 
	the contribution from $J_2$ will be $O\big(n^{\frac{2p_2-r_1-r_2}{2}}\big)$ when $(j^{2}_{r_1+r_2+1}, j^{2}_{r_1+ r_2+2}, \ldots, j^{2}_{2p_2})$ is {\it odd-even pair matched}. Note that the constraint of $J_2$ is automatically satisfied because we have already considered $\sum_{s=1}^{r_1} (-1)^s x^{1}_s = 0 \;{(\mbox{mod $n$})}$ and $\sum_{s=1}^{r_2} (-1)^s x^{2}_s = 0 \;{(\mbox{mod $n$})}$. That is why, $(-1)$ is not coming in the contribution of $J_2$. So in this case cardinality of $B^{'}_{P_3}$  will be   
	\begin{align*}
	&O\big(n^{(r_1+  p_1-\frac{r_1}{2}-1) + (r_2 + p_3-\frac{r_2}{2}-1) + ( p_2-\frac{r_1+r_2}{2})}\big) 
	= O\big(n^{ p_1+p_2+p_3-2}\big).
	\end{align*} 
	\noindent \textbf{Case II.} Each $r_i$ is non-zero for all $i = 1, 2, 3.$ Therefore we can represent $J_i$ as
	\begin{align*}
	J_1 & = (x^{1}_1, x^{1}_2, \ldots, x^{1}_{r_1}, x^{3}_1, x^{3}_2, \ldots, x^{3}_{r_3}, j^{1}_{r_1+r_3+1}, j^{1}_{r_1+r_3+2}, \ldots, j^{1}_{2p_1}), \\
	J_2 &= (x^{1}_1, x^{1}_2, \ldots, x^{1}_{r_1}, x^{2}_1, x^{2}_2, \ldots, x^{2}_{r_2}, j^{2}_{r_1+r_2+1}, j^{2}_{r_1+ r_2+2}, \ldots, j^{2}_{2p_2}), \\
	J_3 & = (x^{2}_1, x^{2}_2, \ldots, x^{2}_{r_2}, x^{3}_1, x^{3}_2, \ldots, x^{3}_{r_3}, j^{3}_{r_2+r_3+1}, j^{3}_{r_2+r_3+2}, \ldots, j^{3}_{2p_3}),
	\end{align*}
where 
	$ r_1 + r_3 \leq  2p_1,\ r_1 + r_2\leq 2p_2, \ r_2+r_3\leq 2p_3.$
	
Observe that each entry in $\displaystyle{\cup_{i=1}^{3} S_{J_i} }$ has multiplicity at least two. 
	Now by the similar arguments, as given in Case I, we shall calculate the contributions due to $J_i$. For the maximum contribution, $(2p_1-r_1-r_3)$ many entries in $J_1$ have to be {\it odd-even pair matched}, as we discussed in Case I. So the maximum contribution from $J_1$ will be $O\big(n^{(r_1+r_3+ \frac{2p_1-r_1-r_3}{2}-1)}\big)$, where $(-1)$ is coming due to the new constraint on the entries of $J_1$, that is, $\sum_{s=1}^{r_1}(-1)^s x^{1}_s + \sum_{s=1}^{r_3}(-1)^{r_1+s} x^{3}_s = 0 \;{(\mbox{mod $n$})}$. Now after fixing entries in $J_1$, by similar argument the maximum contribution from $J_3$ will be $O\big(n^{(r_2 + \frac{2p_3-r_2-r_3}{2}-1)}\big)$, where $(-1)$ is coming due to the new constraint on the entries of $J_3$, that is, $\sum_{s=1}^{r_2}(-1)^sx^{2}_s + \sum_{s=1}^{r_3}(-1)^{r_2+s} x^{3}_s= 0 \;{(\mbox{mod $n$})}$. After selecting $J_1, J_3;$ 
	by similar argument, the contribution from $J_2$ will be $O\big(n^{( \frac{p_2-r_1-r_2}{2})}\big)$. So in this case the cardinality of $B^{'}_{P_3}$  will be   
	\begin{align*}
	&O\big(n^{(r_1+r_3 + p_1-\frac{r_1+r_3}{2}-1) + (r_2 + p_3-\frac{r_2+r_3}{2}-1) + ( p_2-\frac{r_1+r_2}{2})}\big) 
	= O\big(n^{ p_1+p_2+p_3-2}\big).
	\end{align*} 
	Therefore after combining the Case I and Case II, we get
	$$|B^{'}_{P_3}|=  O\big(n^{ p_1+p_2+p_3-2}\big)= o\big(n^{ p_1+p_2+p_3-\frac{3}{2}}\big).$$
	This show that (\ref{equation:cluster1}) is true for $\ell =3.$ So Lemma \ref{lem:cluster} is true for $\ell=3$.
	
	Now for $r \neq 0$, one can show that $|B^{''}_{P_3}| = O\big(n^{ p_1+p_2+p_3-\frac{4+r}{2}}\big)$, which is less than the cardinality of $B^{'}_{P_3}$. We leave it to the readers to verify.
	
	 Suppose (\ref{equation:cluster1}) is true for $\ell = k$, that is, 
	$$|B^{'}_{P_k }| = o \big(n^{p_1+p_2 + \cdots + p_k -\frac{k}{2} }\big).$$
Now we shall prove it for $\ell=k+1$.	
	Suppose  $ (J_1, J_2, \ldots,J_k, J_{k+1}) \in B_{P_{k+1}} $ with $J_i = (j^{i}_1, j^{i}_2, \ldots, j^{i}_{2p_i} )$ for $i=1, 2, \ldots, k+1$, then $ \{J_1, J_2, \ldots, J_{k+1}\} $ forms a cluster. First consider $J_1$, since $ \{ J_1, J_2, \ldots, J_{k+1} \} $ form a cluster therefore there exists at least one $J_i \in \{ J_2, J_3, \ldots, J_{k+1} \}$ such that $J_i$ is connected with $J_1$, and without loss of generality, we suppose $ J_i = J_2$. So $J_2$ is connected with $J_1$. Again as $ \{ J_1, J_2, \ldots,  J_{k+1} \} $ forms a cluster, there exists at least one $J_i \in \{ J_3, J_4, \ldots, J_{k+1} \}$ such that $J_i$ is connected either with $J_1$ or $J_2$ or both, and without loss of generality, we assume that $ J_i = J_3$. Similarly we can arrange $ \{ J_4, J_5, \ldots, J_k \}$. Note that, due to the above arrangement, $\{ J_1, J_2, \ldots, J_k \}$ forms a cluster. But it is given that  $ \{ J_1, J_2, \ldots, J_k, J_{k+1} \} $ forms a cluster, therefore $J_{k+1}$ has to be connected with at least one of $\{ J_1, J_2, \ldots, J_k \}$. Note that, here $\{ J_1, J_2, \ldots, J_k, J_{k+1} \}$ are arranged in such a way that if we remove $J_{k+1}$ from the cluster formed by $\{ J_1, J_2, \ldots, J_k, J_{k+1} \}$, then $\{ J_1, J_2, \ldots, J_k \}$ also form a cluster.
	
	Now we consider $B^{'}_{P_{k+1}}$, that is, there is no element in $\displaystyle{\cup_{i=1}^{k+1} S_{J_i} }$ which has \textit{cross-multiplicity} more than two. Suppose for $i=1, 2, \ldots, k$; $r_i$ many entries of $J_{k+1}$ matches with $r_i$ many entries of $J_i$ only. 
	 Therefore we can represent $J_i$ as $$J_i = (x^{i}_1, x^{i}_2, \ldots, x^{i}_{r_i}, j^{i}_{r_i+1}, j^{i}_{r_i+2}, \ldots, j^{i}_{2p_i} ) \mbox{ for } i= 1, 2, \ldots, k,$$ and 
	
	\begin{equation} \label{eqn:J_{k+1}}
	J_{k+1} = (x^{1}_1, x^{1}_2, \ldots, x^{1}_{r_1}, x^{2}_{1}, x^{2}_{2}, \ldots, x^{2}_{r_2}, \ldots, x^{k}_{1}, \ldots, x^{k}_{r_k}, j^{k+1}_{r'+1}, j^{k+1}_{r'+2}, \ldots, j^{k+1}_{2p_{k+1}} ),
	\end{equation}
	where
	\begin{equation*}
	0 \leq r_i \leq 2p_i, \  \sum_{i=1}^{k} r_i \leq 2p_{k+1} \ \mbox{and} \ r' = \sum_{i=1}^{k}r_i.
	\end{equation*}
	Observe that $r_i$ is non-zero for at least one $i$, as $\{ J_1, J_2, \ldots, J_{k+1} \}$ form a cluster. Now we define a multi-set $X$ as
	\begin{align} \label{def:X}
	X & = \{  x^{1}_1, x^{1}_2, \ldots, x^{1}_{r_1}, x^{2}_1, x^{2}_2, \ldots, x^{2}_{r_2}, \ldots, x^{k}_1, x^{k}_2, \ldots, x^{k}_{r_k}\}.
	\end{align}
	\begin{figure}[h]
		\centering\vskip-10pt
		\includegraphics[height=70mm, width =60mm ]{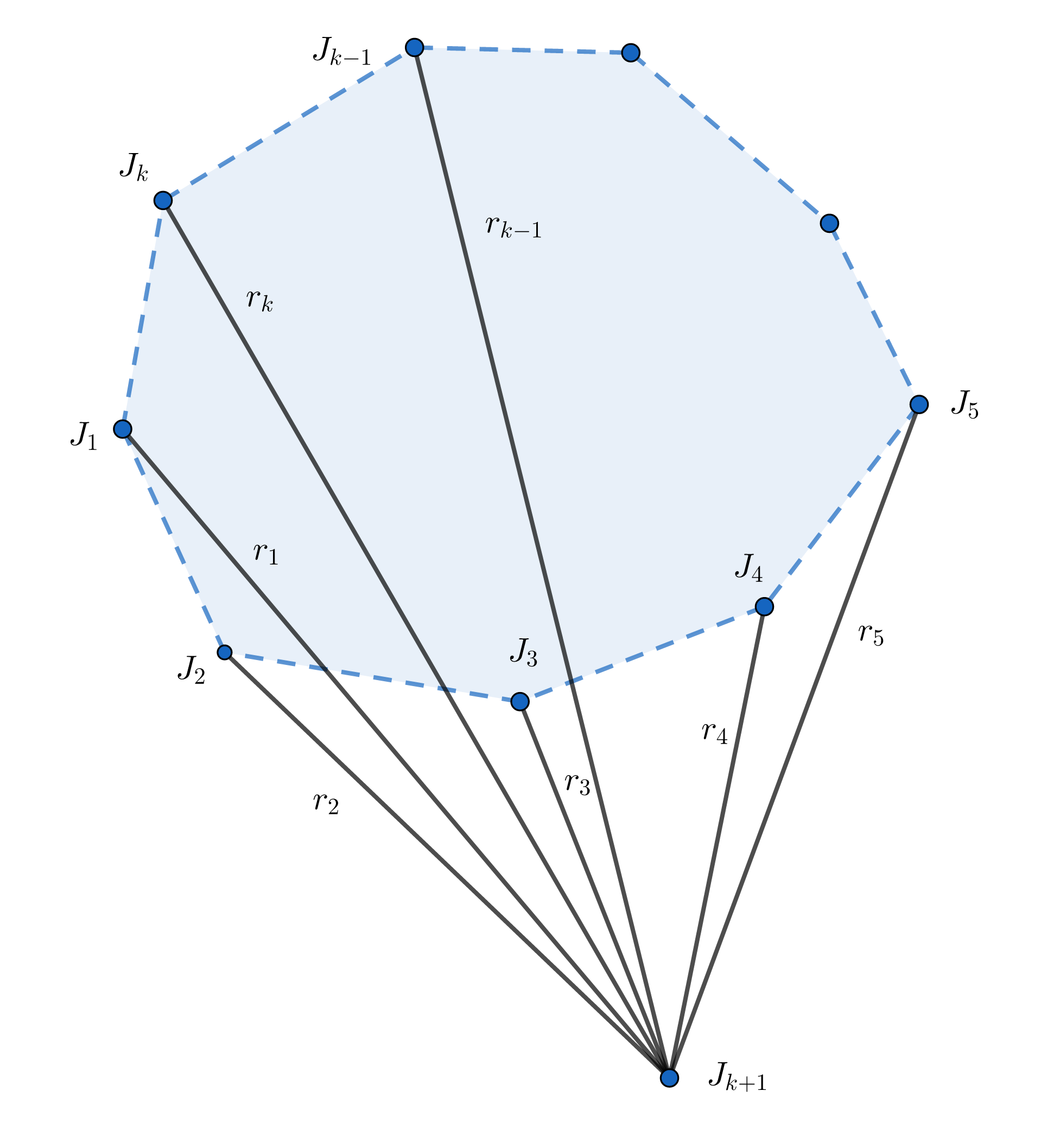}
		\vskip-10pt
		\caption{Connection between $J_1,J_2, \ldots J_k, J_{k+1}$}	
	\end{figure}
First note that $r_i \neq 2p_i$ for all $i=1, 2, \ldots, k.$ Otherwise we will get an element in $\displaystyle{\cup_{i=1}^{k+1} S_{J_i} }$ which has \textit{cross-multiplicity} greater than  two, as $\{ J_1, J_2, \ldots, J_{k} \}$ forms a cluster after removing $J_{k+1}$. That is not possible as we are considering the set $B^{'}_{P_{k+1}}$.
    
     Now let us calculate the contributions of $J_i$'s to the cardinality of $B^{'}_{P_{k+1}}$, as we have calculated for $\ell = 3$. Note that, we will get maximum number of free entries in $B^{'}_{P_{k+1}}$, when each $r_i$ will be even, as we have assumed for $\ell=3$ case and we have proved this in the Case I and Case II of proof of Theorem \ref{thm:revcircovar}. We first calculate the contribution from $J_{k+1}$. We shall get maximum contribution from $J_{k+1}$, if the following holds:
      \begin{enumerate}
     	\item [(i)] each elements of $X$ has to be distinct,
     	\item [(ii)] $X$ have no common element with $\{j^{k+1}_{r'+1}, j^{k+1}_{r'+2}, \ldots, j^{k+1}_{2p_{k+1}}\}$,
     	\item [(iii)]  $(j^{k+1}_{r'+1}, j^{k+1}_{r'+2}, \ldots, j^{k+1}_{2p_{k+1}})$ has to be {\it odd-even pair matched}, 
     \end{enumerate}
          where $r' =\sum_{i=1}^{k} r_i \leq 2p_{k+1}$. Recall that, $J_{k+1}$ has a constraint, namely,  
    $$ \sum_{s=1}^{r_1}(-1)^s x^{1}_s + \sum_{s=1}^{r_2}(-1)^{r_1+s}x^{2}_s+ \cdots + \sum_{s=1}^{r_k}(-1)^{r_1+r_2+\cdots +r_{k-1}+s} x^{k}_s+ \sum_{s=r'+1}^{2p_{k+1}}(-1)^s j^{k+1}_{s} = 0 \;{(\mbox{mod $n$})}.$$  
    Since $(j^{k+1}_{r'+1}, j^{k+1}_{r'+2}, \ldots, j^{k+1}_{2p_{k+1}})$ is {\it odd-even pair matched} and each elements of $X$ are distinct. Therefore the above constraint of $J_{k+1}$ will change into the new form
    \begin{equation}\label{eqn:J_k+1 constrain}
    \sum_{s=1}^{r_1}(-1)^sx^{1}_s + \sum_{s=1}^{r_2}(-1)^{r_1+s}x^{2}_s+ \cdots + \sum_{s=1}^{r_k}(-1)^{r_1+r_2+\cdots +r_{k-1}+s}x^{k}_s = 0 \;{(\mbox{mod $n$})}.
    \end{equation} 
    So the maximum contribution from $J_{k+1}$ is
    \begin{equation} \label{eqn:T_{k+1}}
    T_{k+1} = O\big(n^{r'+ \frac{2p_{k+1}-r'} {2}-1}\big),
    \end{equation}
    where $(-1)$ is coming in the expression of $T_{k+1}$ because $J_{k+1}$ has the constraint (\ref{eqn:J_k+1 constrain}).
    After selecting the entries of $J_{k+1}$, at least $r_i$ many entries of $J_i$ are also selected (fixed) and hence effective entries (remains to count) from $J_i$ is at most $(2p_i-r_i)$ for $i=1, 2, \ldots, k$. Note that if $ \{x^{i}_1, x^{i}_2, \ldots, x^{i}_{r_i}\}$ have no common element with $\{j^{i}_{r_i+1}, j^{i}_{r_i+2}, \ldots, j^{i}_{2p_i}\}$ for $i=1, 2, \ldots, k$, then after selecting the entries of $J_{k+1}$, exactly $r_i$ many entries of $J_i$ are also selected (fixed) and hence effective entries (remains to count) from $J_i$ is exactly $(2p_i-r_i)$ for $i=1, 2, \ldots, k$.
    Observe that the constraint on $J_i$, $ \sum_{s=1}^{r_i}(-1)^s x^{i}_s + \sum_{s=r_i+1}^{2p_i} (-1)^sj^{i}_s=0\;{(\mbox{mod $n$})}$ will change to $$\sum_{s=r_i+1}^{p_i} (-1)^s j^{i}_s= - d_{r_i}\;{(\mbox{mod $n$})},$$ where $ d_{r_i} = \sum_{s=1}^{r_i} (-1)^s x^{i}_s$.
    Since $|B^{'}_{P_{k+1}}|$ is product of $T_{k+1}$ and the number of free entries of $(J_1, J_2, \ldots, J_k)$ which remains to count after selecting $J_{k+1}$, that is, $ |\tilde{B}_{P_{k}}|$. Therefore the cardinality of $B^{'}_{P_{k+1}}$ will be
    \begin{equation} \label{eqn:B_{P_{k+1}}}
    |B^{'}_{P_{k+1}}| = |\tilde{B}_{P_{k}}| T_{k+1},
    \end{equation} 
    where $\tilde{B}_{P_{k}}$ is defined as,
    \begin{align}\label{def:B tilde}
    \tilde{B}_{P_{k}} =  \{ (J_1, J_2, \ldots, J_k) \in B^{'}_{P_k}  :& \ r_i \mbox{ many entries of } J_i  \\
    &	 \mbox{ are already chosen for } i=1,2, \ldots, k \}.\nonumber
    \end{align}
    Now we calculate the cardinality of $\tilde{B}_{P_{k}}$. 
    We denote $( j^{i}_{r_i+1}, j^{i}_{r_i+2}, \ldots, j^{i}_{p_i} )$ by $J'_i$ for $i=1, 2, \ldots, k.$
    Observe that, due to the arrangement of $J_1, J_2, \ldots, J_k,J_{k+1}$, if we remove $J_{k+1}$ from $B^{'}_{P_{k+1}}$,
    then $\{ J_1, J_2, \ldots, J_k \}$ also form a cluster and hence $\{ J'_1, J'_2, \ldots, J'_k\}$ also form a cluster. Since each element of $\displaystyle{\cup_{i=1}^{k+1} S_{J_i} }$ has \textit{cross-multiplicity} less than or equal to two, therefore each element of $\displaystyle{\cup_{i=1}^{k} S_{J'_i} }$ also has \textit{cross-multiplicity} less than or equal to two.
    As for the maximum contribution, $ \{x^{i}_1, x^{i}_2, \ldots, x^{i}_{r_i}\}$ have no common element with $\{j^{i}_{r_i+1}, j^{i}_{r_i+2}, \ldots, j^{i}_{p_i}\}$ for all $i=1, 2, \ldots, k$ and since each elements of $\displaystyle{\cup_{i=1}^{k+1} S_{J_i} }$ has multiplicity greater than or equal to two, therefore each element of $ \displaystyle{\cup_{i=1}^{k} S_{J'_i} }$ also has multiplicity greater than or equal to two.
    From the above discussion on $\{ J'_1, J'_2,  \ldots, J'_k\}$, we get
    \begin{enumerate}
    	\item [(i)] $\{ J'_1, J'_2,  \ldots, J'_k\}$ form a cluster,
    	\item  [(ii)] each element of $\displaystyle{\cup_{i=1}^{k} S_{J'_i} }$ has multiplicity greater than or equal to two,
    	\item  [(iii)] each element of $\displaystyle{\cup_{i=1}^{k} S_{J'_i} }$ has \textit{cross-multiplicity} less than or equal to two,
    	\item [(iv)] $j^{i}_{r_i+1}+ j^{i}_{r_i+2}+ \cdots+ j^{i}_{p_i}= -d_{r_i}\;{(\mbox{mod $n$})}$ for $i=1, 2, \ldots ,k.$
    \end{enumerate}
    So the set $\tilde{B}_{P_{k}}$ is the collection of the vectors $ (J'_1, J'_2,  \ldots, J'_k)$, which satisfy the above four conditions.
    Now consider the subset $ B^{'}_{P'_k}$ of $ B_{P'_k}$, where $B_{P'_k}$ is as in the definition \ref{def:B_{P_l}} with $P'_k = (p'_1, p'_2, \ldots, p'_k)$ and $p'_i = 2p_i -r_i$ for $i =1, 2, \ldots, k$.
 Therefore   
    \begin{align}\label{def:B' tilde}
    B^{'}_{P'_k} =  \{ (T_1, T_2,  \ldots, T_k) \in B_{P'_k}  : & \mbox{ \textit{cross-multiplicity} of each entry of }\cup_{i=1}^{k} S_{T_i}  \\
    &	\mbox{ is less than or equal to two}\}. \nonumber
    \end{align}
  Now  by the induction hypothesis, we get
    $$|B^{'}_{P^{'}_{k} }| = o \big(n^{ p^{'}_{1}+p^{'}_{2} + \cdots + p^{'}_k - \frac{k}{2}} \big).$$
    Now we compare the cardinality of $\tilde{B}_{P_{k}}$ and $B^{'}_{P'_{k} }$. Note that both the sets $\tilde{B}_{P_{k}}$ and $B^{'}_{P'_{k} }$ are almost same, the only difference is in the constraints. Recall that, if $(J'_1, J'_2, \ldots, J'_k) \in \tilde{B}_{P_{k}}$, then entries of $J'_i$ has constraints 
    \begin{equation} \label{eqn:constraint non-zero}
    j^{i}_{r_i+1}+ j^{i}_{r_i+2}+ \cdots+ j^{i}_{p_i}= - d_{r_i}\;{(\mbox{mod $n$})} \mbox{ for } i=1, 2, \ldots ,k,
    \end{equation} 
    where $d_{r_i}$ are constants for $i=1, 2, \ldots, k$. 
    
    Recall that, if $(T_1, T_2, \ldots, T_k) \in B^{'}_{P'_{k} }$, then entries of $T_i$ has constraints  
    \begin{equation} \label{eqn:constraint zero}
    t^{i}_{1}+ t^{i}_{2}+ \cdots+ t^{i}_{p'_i}= 0\;{(\mbox{mod $n$})} \mbox{ for } i=1, 2, \ldots ,k.
    \end{equation} 
    Note that, if $d_{r_i} = d_{r_1} \;{(\mbox{mod $n$})} \mbox{ for } i=2, 3, \ldots ,k,$ then the constraints of $\tilde{B}_{P_{k}}$, that is, (\ref{eqn:constraint non-zero}) will be 
    $$j^{i}_{r_i+1}+ j^{i}_{r_i+2}+ \cdots+ j^{i}_{p_i}= - d_{r_1}\;{(\mbox{mod $n$})} \mbox{ for } i=1, 2, \ldots ,k, $$
    which shows that $|\tilde{B}_{P_{k}}| =|B^{'}_{P'_{k} }|$. 
    If $d_{r_i} \neq d_{r_1} \;{(\mbox{mod $n$})} \mbox{ for at least one } i,$ then the entries of $\tilde{B}_{P_{k}}$ have to satisfy more constraints than the constraints on entries of $B^{'}_{P'_{k} }$, that is, $|\tilde{B}_{P_{k}}| \leq |B^{'}_{P'_{k} }|$. Hence 
    \begin{equation} \label{eqn:tilde B B'}
    |\tilde{B}_{P_{k}}| \leq |B^{'}_{P'_{k} }|.
    \end{equation}
    Therefore from (\ref{eqn:B_{P_{k+1}}}) and (\ref{eqn:tilde B B'}), we get
    \begin{align*}
    |B^{'}_{P_{k+1}}|& = |\tilde{B}_{P_{k}}| T_{k+1} \\
    & \leq |B^{'}_{P^{'}_{k}}|  T_{k+1} \\
    &= o \big(n^{p^{'}_{1} +p^{'}_{2} + \cdots + p^{'}_k - \frac{k}{2}} \big) O\big(n^{r'+ \frac{2p_{k+1}-r'}{2}-1 }\big) \\
    &= o \big(n^{( p^{'}_{1} +p^{'}_{2} + \cdots + p^{'}_k - \frac{ k}{2}) + (r'+ \frac{2p_{k+1}-r'}{2}-1 ) }  \big).
    \end{align*}
    Since $p^{'}_i = 2p_i-r_i$ for $i=1, 2, \ldots ,k$ and $ \sum_{i=1}^{k} r_i = r'$, from the last equation, we have
    \begin{align*}
    |B^{'}_{P_{k+1}}| & = o\big(n^{p_1+p_2 + \cdots + p_k + p_{k+1} -\frac{ k}{2}-1}\big)
    =  o\big(n^{p_1+p_2 + \cdots + p_k + p_{k+1} -\frac{k+1}{2} }\big).
    \end{align*}
    Therefore from (\ref{eqn:B_P,B_P'}), we get
    $$|B_{P_{k+1}}| = o\big(n^{p_1+p_2 + \cdots + p_k + p_{k+1} -\frac{ k+1}{2} }\big).$$
    So Lemma \ref{lem:cluster} is true for $\ell=k+1$. This complete the proof of Lemma \ref{lem:cluster}.
\end{proof}


\providecommand{\bysame}{\leavevmode\hbox to3em{\hrulefill}\thinspace}
\providecommand{\MR}{\relax\ifhmode\unskip\space\fi MR }
\providecommand{\MRhref}[2]{%
  \href{http://www.ams.org/mathscinet-getitem?mr=#1}{#2}
}
\providecommand{\href}[2]{#2}

\end{document}